\documentclass{amsart}

\usepackage{amsmath}
\usepackage{amssymb}
\usepackage{nccmath}
\usepackage{graphicx}
\usepackage{amsfonts}
\usepackage{hyperref}
\usepackage[numbers]{natbib}
\usepackage{lipsum}
\usepackage{mathrsfs}
\usepackage{epstopdf}
\usepackage{multirow}
\usepackage{enumitem}
\usepackage[utf8]{inputenc}
\usepackage{mathtools} 
\usepackage{accents}
\usepackage{commath}
\numberwithin{equation}{section}
\usepackage{amsthm}
\newtheorem{thm}{Theorem}
\numberwithin{thm}{section}
\newtheorem{lem}[thm]{Lemma}

\newtheorem{prop}[thm]{Proposition}

\newtheorem{definition}[thm]{Definition}

\usepackage{amscd,amstext}
\usepackage{comment}
\usepackage{tikz}
\usetikzlibrary{calc}
\usepackage{caption}  
\theoremstyle{definition}
\newtheorem{proofpart}{Step}
\usepackage{verbatim}
\usepackage{nccmath}
\newtheorem*{structhyp1}{Structural Hypothesis (SH)}
\usepackage{url}
\usepackage{doi}

\def\R{\mathbb{R}}

\def\D{\mathrm{D}}
\def\Om{\Omega}

\def\p{\mathrm{p}}

\newcommand{\mC}{\mathcal{C}}

\newcommand{\mH}{\mathcal{H}}

\newcommand{\mD}{\mathcal{D}}

\newcommand{\noi}{\noindent}

\newcommand{\al}{\alpha}
\newcommand{\be}{\beta}

\newcommand{\de}{\delta}

\newcommand{\e}{\varepsilon}

\newcommand{\Si}{\Sigma}
\newcommand{\la}{\lambda}

\newcommand{\weak }{\, -\!\!\!\!\!-\!\!\!\!\rightharpoonup}

\newcommand{\larrow}{\longrightarrow}
\newcommand{\ot}{\otimes}

\newcommand{\LL}{\text{\LARGE$\llcorner$}}

\renewcommand{\p}{\partial}

\newcommand{\sub}{\subseteq}

\newcommand{\by}{\times}

\newcommand{\beq}{\begin{equation}}

\newcommand{\eeq}{\end{equation}}

\makeatletter
\@namedef{subjclassname@2020}{\textup{2020} Mathematics Subject Classification}
\makeatother

\begin{document}

\title{On General Linear Degenerate Elliptic PDE Systems}

\date{}

\author{Nikos Katzourakis and Frederick Temple}



\address{N.K. (corresponding author) Department of Mathematics and Statistics, University of Reading, Whiteknights Campus, Pepper Lane, Reading RG6 6AX, UNITED KINGDOM}

\address{F.T. Department of Mathematics and Statistics, University of Reading, Whiteknights Campus, Pepper Lane, Reading RG6 6AX, UNITED KINGDOM}

\thanks{N.K.\ has been partially financially supported through the EPSRC grant EP/X017206/1.}

\maketitle

\setcounter{section}{0}
\begin{abstract} 
\noindent
Let $\Omega \Subset \mathbb{R}^{n}$ be a strictly convex bounded domain. Suppose $\mathbf{A} : \mathbb{R}^{Nn} \longrightarrow \mathbb{R}^{Nn}$, $\mathbf{B}: \mathbb{R}^{Nn} \longrightarrow \mathbb{R}^{N}$, $\mathbf{C}: \mathbb{R}^{N} \longrightarrow \mathbb{R}^{N}$ are linear maps, where $\mathbf{A}$ is symmetric and non-negative definite. Given $f \in L^2(\Omega, \mathbb{R}^N)$, we consider the problem of existence of solutions $u: \Omega  \longrightarrow \mathbb{R}^N$ to the PDE system
\[
 \left\{
\begin{array}{rl}
\displaystyle\sum_{\beta = 1}^{N}\sum_{i, j = 1}^{n} \mathbf{A}_{\alpha i \beta j}\mathrm{D}_{ij}^{2}u_{\beta} + \sum_{\beta = 1}^{N} \sum_{i=1}^{n} \mathbf{B}_{\alpha \beta i}\mathrm{D}_{i}u_{\beta} + \sum_{\beta = 1}^{N} \mathbf{C}_{\alpha \beta}u_{\beta} = f_\alpha, &\text{ in $\Omega$},
\\
u = 0,\ \,& \text{ on $\partial \Omega$}.
\end{array} 
\right. 
\]
This is a linear \textit{degenerate elliptic} system, and it has not been considered before without the assumption of strict rank-one convexity. In general, it may not possess not even distributional solutions. By introducing some natural structural assumptions, we prove the existence of an appropriately defined unique generalised solution $u\in L^2(\Omega, \mathbb{R}^N)$, satisfying additional partial regularity properties. This paper extends earlier work of the first appearing author [\textit{N. Katzourakis, On linear degenerate elliptic PDE systems with constant coefficients}, Adv.\ in Calc.Var.\ 9:3, 283-291 (2016)] to include lower-order terms. 
\end{abstract}

\section{Introduction and the Main Result}

\noindent
Let $n, N \geq 1$ be integers, and suppose that $\mathbf{A} : \mathbb{R}^{Nn} \longrightarrow \mathbb{R}^{Nn}$, $\mathbf{B}: \mathbb{R}^{Nn} \longrightarrow \mathbb{R}^{N}$, $\mathbf{C}: \mathbb{R}^{N} \longrightarrow \mathbb{R}^{N}$ are linear mappings. Let us further assume $\mathbf{A}$ defines a convex symmetric quadratic form on $\mathbb{R}^{Nn}$, i.e.,
\begin{equation} \label{eq:0}
\mathbf{A}_{\alpha i \beta j} = \mathbf{A}_{\beta j \alpha i}, \hspace{0.3cm} \sum_{\alpha, \beta, i, j} \mathbf{A}_{\alpha i \beta j} Q_{\alpha i} Q_{\beta j} \geq 0, \hspace{0.3cm} \forall\ Q \in \mathbb{R}^{Nn}.
\end{equation}
Throughout this paper, we will adhere to the convention that Greek indices $\alpha, \beta, \gamma$,  $\ldots$ run in $\{1, \dotsc, N \}$ and Latin indices $i, j, k, \ldots$ run in $\{1, \dotsc, n \}$, even if their range is not denoted explicitly. The symbols $\{e^i \}$, $\{e^\alpha \}$, and $\{e^\alpha \otimes e^i \}$ will represent the respective standard bases of $\mathbb{R}^n$, $\mathbb{R}^N$ and $\mathbb{R}^{Nn}$.  Let also $\Omega \Subset \mathbb{R}^N$ be a bounded strictly convex domain, henceforth fixed. Strict convexity herein is understood in the sense that any (affine) tangent space $x+T_x\p\Omega$ of the boundary $\p\Omega$ intersect the boundary only at $x \in \p\Omega$. We are interested in the problem of existence and uniqueness of solutions $u : \Omega \longrightarrow \mathbb{R}^N$ to the Dirichlet problem for the following linear PDE system
\begin{equation} \label{eq:1}
 \left\{
\begin{array}{rl}
\displaystyle\sum_{\beta = 1}^{N}\sum_{i, j = 1}^{n} \mathbf{A}_{\alpha i \beta j}\mathrm{D}_{ij}^{2}u_{\beta} + \sum_{\beta = 1}^{N} \sum_{i=1}^{n} \mathbf{B}_{\alpha \beta i}\mathrm{D}_{i}u_{\beta} + \sum_{\beta = 1}^{N} \mathbf{C}_{\alpha \beta}u_{\beta} = f_\alpha, &\text{ in $\Omega$}, 
\\
u = 0,\ \,& \text{ on $\partial \Omega$},
\end{array} 
\right. 
\end{equation}
when $f \in L^2(\Omega, \mathbb{R}^N)$ is also given. Despite the system being linear with constant coefficients, it is {\it degenerate elliptic}, namely we do {\it not} assume either that $ \mathbf{A}$ is a strictly convex quadratic form on $\mathbb R^{N\times n}$, or even that it is strictly rank-one convex, which would amount to requiring
\[
 \sum_{\alpha, \beta, i, j} \mathbf{A}_{\alpha i \beta j} \eta_{\alpha} a_i \eta_{\beta} a_j > 0, \hspace{0.3cm} \forall\ \eta \in \mathbb{R}^{N}\setminus \{0\},\ \forall\ a \in \mathbb{R}^{n}\setminus \{0\}.
\]
In the absence of strict ellipticity, the system may not in general possess a solution in any reasonable sense, unless certain compatibility conditions are satisfied. For example, the next $2\times n$ Poisson-type system
\[
\left\{ 
\begin{array}{rl}
\Delta u_1 =f_1 , & \text{ in }\Omega,
\\
0 =f_2 , & \text{ in }\Omega,
\\
u=0, \ & \text{ on }\partial \Omega,
\end{array}
\right.
\]
arising via the admissible choices $\mathbf{A}_{\alpha i\beta j}=\delta_{\alpha 1}\delta_{1 \beta} \delta_{ij}$, $\mathbf{B}\equiv 0$ and $\mathbf{C}\equiv 0$,  has no solution whatsoever for those $f \in L^2(\Omega, \mathbb{R}^2)$ for which $f_2\not\equiv 0$. Therefore, we need to develop a theory taking into account the degeneracies of the leading coefficient $\mathbf{A}$, with appropriate compatibility conditions on all the coefficients. To this aim, we define the following vector spaces

\begin{equation} \label{eq:01}
\left\{ \ \ \begin{split}
\Pi & \coloneqq \mathrm N\big(\mathbf{A} :\mathbb{R}^{Nn} \longrightarrow \mathbb{R}^{Nn}\big)^{\perp} \hspace{67pt} \subseteq \mathbb{R}^{Nn},
\\
\Sigma &\coloneqq \text{span} [\Big\{\eta \in \mathbb{R}^N \ \big| \ \exists \, a \in \mathbb{R}^n : \eta \otimes a \in \Pi \Big\}] \subseteq \mathbb{R}^{N}.
\end{split}
\right.
\end{equation}
Here $\Pi$ is the orthogonal complement of the nullspace of $\mathbf{A}$, and $\Sigma$ is the linear span of the projections of rank-one directions onto the physical space $\mathbb R^N$, wherein the solutions are valued, of the set of all rank-one directions contained in $\Pi$. The ``size" of $\Pi$ determines the degree of degeneracy of the PDE system: for any $\eta \otimes a \in \Pi$, rank-one convexity is true along this direction, whilst $\Pi = \mathbb{R}^{Nn}$ if and only if $\mathbf{A}$ defines a strictly convex quadratic form (and then $\Sigma = \mathbb{R}^N$).  The larger the space $\Pi$ is, the larger the set of non-degenerate (rank-one) directions it contains is, along which rank-one convexity succeeds.

In order to adapt our functional spaces to the degeneracies of the system, spaces of mappings $f : \Omega \longrightarrow \Sigma \subseteq \mathbb{R}^N$, will be denoted by $C^\infty(\Omega, \Sigma)$, $L^p(\Omega, \Sigma)$, $W^{1, p}(\Omega, \Sigma)$ etc.  We will also use the convention that the same letters $\Pi, \Sigma$ will represent their respective subspaces as well as the orthogonal projections onto them. Let 
\begin{equation*}
\mathscr{D}(\Omega, \Sigma) \coloneqq C^\infty_c(\Omega, \Sigma)
\end{equation*}
be the space of test functions valued in $\Sigma$; we will denote the dual space as 
 \begin{equation*}
\mathscr{D}'(\Omega, \Sigma) \coloneqq (\mathscr{D}(\Omega, \Sigma))^*.
\end{equation*}
Both $\mathscr{D}$ and $\mathscr{D}'$ are considered to have their respective standard topologies (as e.g.\ defined in \cite{folland1} in the Euclidean-valued case), which however will not be explicitly required to be discussed. The following is the central notion of generalised solution we will utilise in this work.

\begin{definition}[Adapted distributional solutions]
A mapping $u \in \mathscr{D}'(\Omega, \Sigma)$ will be called an \emph{adapted distributional solution} to the PDE system
\begin{equation*}
\sum_{\beta, i, j} \mathbf{A}_{\alpha i \beta j}\mathrm{D}_{ij}^{2}u_{\beta} + \sum_{\beta, i} \mathbf{B}_{\alpha \beta i}\mathrm{D}_{i}u_{\beta} + \sum_{\beta} \mathbf{C}_{\alpha \beta}u_{\beta} = f_{\alpha} \hspace{0.1cm} \text{ in $\Omega$},
\end{equation*}
when $u \in L^1_{loc}(\Omega, \mathbb{R}^N)$ and, for all $\phi \in \mathscr{D}(\Omega, \Sigma)$, we have
\[
\begin{split}
\int_{\Omega} \sum_{\alpha, \beta, i, j} \mathbf{A}_{\alpha i \beta j} u_{\beta}  \mathrm{D}^2_{ij}\phi_{\alpha} &- \int_{\Omega} \sum_{\alpha, \beta, i} \mathbf{B}_{\alpha \beta i} u_{\beta}\mathrm{D}_{i}\phi_{\alpha} 
+  \int_{\Omega} \sum_{\alpha, \beta} \mathbf{C}_{\alpha \beta}u_{\beta}\phi_{\alpha} = \int_{\Omega} \sum_{\gamma} f_{\gamma} \phi_{\gamma}.
\end{split}
\] 
\end{definition}

Unlike the elliptic case, whereat it can be shown through a priori estimates that any solution is actually strong and in fact weakly twice differentiable a.e.\ on $\Omega$ (see e.g.\ \cite{GM}), in general the solution of the problem at hand cannot be expected to be a Sobolev function, and it may not even once weakly differentiable. To see this, let $\Om=\mathbb B_1\sub \R^2$ be the open unit disc centred at the origin, and choose any function $f\in C(\overline{\mathbb B_1})$ which is non-differentiable weakly with respect to $x_1$, for a.e.\ $x_2$, when $x=(x_1,x_2) \in \mathbb B_1$. Then, one can easily see that the Dirichlet problem
\[
\left\{
\begin{split}
D^2_{22}u \, &=\, f , \ \text{ in }\mathbb B_1,\\
u\,&=\, 0, \ \text{ on }\p\mathbb B_1,
\end{split}
\right.
\]
has the explicit solution
\[
u(x_1,x_2)\, = \int_{-\infty}^{x_2}\int_{-\infty}^{t_2}f(x_1,s_2)\,ds_2\,dt_2 \, -\, h(x_1,x_2),
\]
where $h : \overline{\mathbb B_1} \larrow \R $ and $g : \partial  {\mathbb B_1} \larrow \R$ are given by
\[
\left\{\ \ 
\begin{split}
h(x_1,x_2)\, :=\, &\left(\frac{g\big(x_1, \sqrt{1-x_1^2} \big) - g\big(x_1, -\sqrt{1-x_1^2} \big) } {2\sqrt{1-x_1^2}} \right)x_2\\
&\ \ \ \ \ \ \ \ +\, \frac{g\big(x_1, \sqrt{1-x_1^2} \big) + g\big(x_1, -\sqrt{1-x_1^2} \big) } {2} ,
\\
g(x_1,x_2)\, :=\, &\int_{-\infty}^{x_2}\int_{-\infty}^{t_2}f(x_1,s_2)\,ds_2\,dt_2.
\end{split}
\right.
\]
but due to the regularity properties of $f$, it follows that $u\not\in W^{1,1}_\text{loc}(\Om)$.

As is well known, even in the scalar-valued strictly elliptic case, certain conditions are required to be satisfied by the coefficients of the PDE for the corresponding Dirichlet problem to be solvable (most notably that $\mathbf{C}\leq0$). In our generalised vectorial context, we will need to utilise the following assumptions regarding the tensor coefficients $\mathbf{B}: \mathbb{R}^{Nn} \longrightarrow \mathbb{R}^{N}$ and $\mathbf{C}: \mathbb{R}^{N} \longrightarrow \mathbb{R}^{N}$, which take into account the degenerate structure of $\mathbf{A}$, in order to guarantee solvability (recall that $\Pi,\Si$ are given by \eqref{eq:01}):

\begin{equation} \label{as:2}
\mathbf{C} = \Sigma \mathbf{C} \Sigma= \mathbf{C} \Sigma= \Sigma \mathbf{C}, \ \ \ \ \mathbf{C} = \mathbf{C}^\top \! \leq 0,
\end{equation}
\begin{equation} \label{as:3}
\mathbf{B} = \Sigma \mathbf{B} \Pi = \Sigma \mathbf{B} = \mathbf{B} \Pi, \ \ \ \ \mathbf{B}_{\alpha \beta i} = \mathbf{B}_{\beta \alpha i}.
\end{equation}
These assumptions require commutativity of the operators $ \mathbf{B},  \mathbf{C}$ with the respective projections on the subspaces $\Pi \sub \R^{N\by n}$ and $\Si\sub \R^N$, as defined by  \eqref{eq:01}, together with symmetry in the $\al \leftrightarrow \be$ indices, alongside non-positivity for the matrix $\mathbf{C}$. In the case that $\mathbf{A}$ is strictly elliptic without degeneracies, it follows that $\Pi = \R^{N\by n}$ and $\Si = \R^N$, hence both projections are merely the identity operators on the corresponding spaces, and no commutativity constraint is imposed. The symmetry condition on $\mathbf{B}$ is perhaps a little more mysterious, but it is easy to find example of such tensors satisfying it; for example, any $\mathbf{B}$ of the form 
\begin{equation} \label{examb:1}
\mathbf{B} = \eta \otimes \eta \otimes a,
\end{equation}
with $\eta \otimes a \in \Pi$, satisfies assumption \eqref{as:3}. In fact, if $k \in \{1, \dotsc, m\}$, the linear combination of rank-one tensors of the form 
\begin{equation} \label{examb:2}
\mathbf{B} = \sum^m_{k=1}\lambda^{(k)} \eta^{(k)} \otimes \eta^{(k)} \otimes a^{(k)}
\end{equation}
with $\{\eta^{(k)} \otimes a^{(k)}\}^m_{k=1} \sub \Pi$ and $\{\lambda^{(k)}\}^m_{k=1}\sub \R$, also satisfies assumptions \eqref{as:3}.

Additionally to the necessity for an adapted notion of generalised solution that respects the degenerate nature of the problem, a significant difficulty in the study of this Dirichlet problem is the satisfaction of the boundary condition. In this non-elliptic ultra-low regularity context, in general there is no trace operator that allows to give a rigorous meaning to ``$u=0$ on $\p\Omega$", as $u\not\in W^{1,1}(\Omega,\R^N)$. Interestingly, however, there does exist an adapted trace operator for the solution, as long as $\Pi$ is spanned by rank-one directions. In fact, the satisfaction of the boundary condition is the {\it only} point that the strict convexity of the domain is needed. Furthermore, there also exists a corresponding adapted Poincar\'e inequality, and along any rank-one non-degenerate direction, we also have a corresponding partial regularity property. These last two auxiliary results have been proved in \cite{nikos1}, and are recalled in the next section. They are motivated by developments in the theory of generalised solutions to systems of fully non-linear PDE system and the Calculus of Variations in the space $L^\infty$ (\cite{K1, K2, K4}), and the same is actually true for the main result herein.

\medskip

The following is the main result of this paper:

\begin{thm}[Existence, uniqueness and partial regularity] \label{thm:1}
Let $\Omega \Subset \mathbb{R}^{n}$ be a strictly convex bounded domain. Suppose that  $\mathbf{A} : \mathbb{R}^{Nn} \longrightarrow \mathbb{R}^{Nn}$ is a symmetric linear map satisfying \eqref{eq:0}. Let $\Pi,\Si$ be given by \eqref{eq:01}, and assume that the vector space $\Pi \subseteq \mathbb{R}^{Nn}$ (the orthogonal complement of the nullspace of $\mathbf{A}$) is spanned by rank-one directions. Suppose further that $\mathbf{B}: \mathbb{R}^{Nn} \longrightarrow \mathbb{R}^{N}$ and $\mathbf{C}: \mathbb{R}^{N} \longrightarrow \mathbb{R}^{N}$ are linear maps satisfying assumptions \eqref{as:2}-\eqref{as:3}. 

Then, for any $f \in L^2(\Omega, \Sigma)$, the PDE system
\[ 
\left\{
\begin{array}{rl} 
			\hfil \displaystyle\sum_{\beta, i, j} \mathbf{A}_{\alpha i \beta j}\mathrm{D}_{ij}^{2}u_{\beta} + \sum_{\beta, i} \mathbf{B}_{\alpha \beta i}\mathrm{D}_{i}u_{\beta} + \sum_{\beta} \mathbf{C}_{\alpha \beta}u_{\beta} = f_{\alpha}, &\text{in $\Omega$}, 
			\\
            u = 0,\ \,  & \text{on $\partial \Omega$,}
\end{array}
\right.
\]
has a unique adapted distributional solution $u \in\mD'(\Om,\Si) \cap L^2(\Omega, \Sigma)$.  Additionally, $u$ is $\mH^{n-1}\LL \p \mC$-measurable with respect to the $(n-1)$-dimensional Hausdorff measure on the boundary of any strictly convex open subdomain $\mC \sub \Om$. Therefore, the boundary condition is satisfied in the strong sense: $u=0$ $\mH^{n-1}$-a.e.\ on $\p\Om$. 

Finally, we have the following partial regularity property: projections of the distributional gradient $\D u$ along any rank-one directions of non-degeneracy inside $\Pi$ exist as $L^2$-functions: 
\[
\D u : \eta\ot a = \D_a(\eta \cdot u) \in L^2(\Om), \ \ \ \forall\ \eta \ot a \in \Pi.
\]
\end{thm}

In the above and subsequently, ``$:$" will denote the Euclidean (trace) inner product on $\R^{N\by n}$, namely $A:B:=\mathrm{tr}(A^\top B)$, as well as the corresponding inner product in higher dimensional tensor spaces. Theorem \ref{thm:1} generalises a corresponding result established in \cite{nikos1} by the first appearing author, which applies to the special case of $\mathbf{B}\equiv 0$ and $\mathbf{C}\equiv 0$, namely without any lower order terms. Otherwise, to the best of our knowledge, problems of this type have not been considered before in the literature. Generally speaking, the inclusion of lower order terms in PDE problems is ordinarily a technical only difficulty. However, the situation is considerably different for the specific problem herein. Due to the degenerate nature of the problem, the inclusion of lower order terms causes essential complications. Even the very method followed in \cite{nikos1}, based on a viscosity approximation and the use of strictly elliptic variational second order elliptic systems, cannot be applied in this case. This owes to the fact that, the term involving $\mathbf{B}$ cannot arise variationally.

Note that the generalised solution obtained in Theorem \ref{thm:1}, is valued in the subspace $\Si\sub \R^N$, namely satisfies $\Sigma^{\perp}u \equiv 0$. This is necessary to guarantee uniqueness, because any arbitrary extension of $u$ from $L^2(\Omega, \Sigma)$ to $L^2(\Omega, \mathbb{R}^N)$ is also a solution. For example, if $u$ is an adapted distributional solution to the Dirichlet problem, for {\it any} $h \in L^1_{c}(\Omega)$, the map
\[
\tilde{u} \coloneqq \Sigma u + h\Sigma^\perp , \hspace{0.1cm}
\]
is also a solution to the Dirichlet problem, because the component of $u$ orthogonal to $\Si$ is not ``seen" by the coefficients of the PDE system, and since $h$ is compactly supported, it satisfies the boundary condition as well.

Even though the structural assumptions \eqref{eq:0}, \eqref{as:2} and \eqref{as:3} are easily quantifiable and confirmed, it is far less clear when is the non-degeneracy subspace $\Pi \sub \R^{N\by n}$ spanned by rank-one directions. Clearly, if $N=1$ (scalar valued case) or if $n=1$ (one-dimensional case), then it is automatically satisfied. Further, it is easy to see that not all subspaces $\Pi$ satisfy this property. For instance, the tensor $\mathbb A=A \otimes A$ (i.e.\ $\mathbb A_{\al i \be j} = A_{\al i} A_{\be j}$), where $A \in  \R^{N\by n}$ is fixed, satisfies $\Pi = \mathrm{span}[\{A\}]$, but if $\mathrm{rk}(A)\geq 2$ then $\Pi$ is not spanned by rank-one directions. For the general vector-valued case, let us recall the following structural hypothesis on $\mathbf{A}$,  taken from \cite{nikos1}; The hypothesis (SH) provides a {\it sufficient constructive condition} for $\Pi \sub \R^{N\by n}$ to be spanned by rank-one directions, as required in Theorem \ref{thm:1}.

\begin{structhyp1}(See \cite{nikos1})
The tensor $\mathbf{A}$ satisfies (SH) if it can be decomposed as 
\[
\mathbf{A}_{\alpha i \beta j} = B^1_{\alpha \beta}A^1_{ij} + \cdots + B^N_{\alpha \beta}A^N_{ij}
\]
and:
\begin{enumerate}[label=(\alph*)]
\item The symmetric matrices $\{A^1, \dotsc, A^N\} \subseteq \mathbb{R}^{n \by n}$ are non-negative.
\item The symmetric matrices $\{B^1, \dotsc, B^N\} \subseteq \mathbb{R}^{N \by N}$ are non-negative and have mutually orthogonal ranges. 
\item The eigenspaces of the matrices $A^{\gamma} - \lambda^{\gamma}_1 \mathrm I : \mathbb{R}^n \longrightarrow \mathbb{R}^n$ intersect for all $\gamma = 1, \dotsc, N$ along a common line. Here $\lambda^{\gamma}_1$ denotes the smallest positive eigenvalue of $A^\gamma$.
\end{enumerate}
\end{structhyp1}
This paper is organised as follows. This introduction section is followed by a preliminaries section in which we collect some auxiliary results and observations, and the last section contains the proof of our main result.


\section{Preliminaries}

We note first that our general notation is either standard PDE and variational notation (as e.g.\ in  \cite{E,EG,D,GM}), or otherwise self-explanatory. We begin by recalling two auxiliary results established in \cite{nikos1}, that will be need for the proof of the main result:

\begin{prop}[Degenerate partial Poincar\'e inequality, cf.\ \cite{nikos1}] \label{prop:1}
Let $\Pi$, $\Sigma$ be as in Theorem \ref{thm:1}, and suppose $\Omega \sub \R^n$ is a bounded open domain. Then, there exists a constant $C = C(\Omega, n, N) > 0$ depending only on the diameter of $\Omega$ and on the dimensions $n$, $N$ such that, for any $u \in W^{1,2}_0(\Omega, \mathbb{R}^N)$ we have the estimate

\begin{equation*}
\| \Sigma u \| _{L^2(\Omega)} \leq C \| \Pi \mathrm{D}u \| _{L^2(\Omega)}.
\end{equation*}
\end{prop}
The point of the above result is that, even if we can control merely the projection of $\D u$ onto the subspace $\Pi \sub\R^{N\by n}$ of the mapping in $L^2(\Omega)$, we can still control at least the projection of the mapping itself onto $\Sigma \sub \R^N$ in $L^2(\Omega)$, even though perhaps not the full map in $L^2(\Omega)$. The result can be stated more concisely in the natural setting of ``Fibre Sobolev spaces", which contain mappings partially differentiable along projections as above (see \cite{K4}), but to simplify the presentation we state it merely for standard Sobolev mappings. This suffices, as in this paper it is applied only to our regularised approximate problems, which have weakly differentiable (strong) solutions. We refrain from giving the proof, for which we refer to \cite{nikos1}, but we should nonetheless mention that the rank-one spanning property is used essentially in the proof of this result. It will be used essentially in the proof in order to provide stable bounds for our approximations.

Next we state a result establishing the existence of a degenerate partial trace operator, also taken from \cite{nikos1}. Once again, it is stated only in the setting of the standard Sobolev spaces, as it will only be directly applied to the solutions of the approximating regularised problems.

\begin{prop}[Degenerate partial trace operator, cf.\ \cite{nikos1}] \label{prop:2}
Let $\Omega$, $\Pi$, $\Sigma$ be as in Theorem \ref{thm:1}. Then, there exists a closed $\mathcal{H}^{n-1}$-nullset $E \subseteq \partial \Omega$ such that, for any $\Gamma \Subset \partial \Omega \setminus E$, we can find $C= C(n,\Gamma) >0$:
\[
\| \Sigma u \| _{L^2(\Gamma, \mathcal{H}^{n-1})} \leq C \left(\| \Pi \mathrm{D}u \| _{L^2(\Omega)} + \| \Sigma u \| _{L^2(\Omega)}\right),
\] 
for all $u \in W^{1,2}(\Omega, \mathbb{R}^N)$.
\end{prop}

Proposition \ref{prop:2} will be used in order to establish the satisfaction of the Dirichlet boundary condition for our degenerate elliptic problem, as standard results on the usual trace operator as e.g.\ in \cite{EG} cannot be applied.

\subsection*{Discussion of the method followed in  \cite{nikos1}, versus the methods used herein.} \label{Variationally} In the paper \cite{nikos1}, the method of proof followed was based on a ``vanishing viscosity" approximation by strictly elliptic systems, for which a (unique) approximate solution was established by using variational methods, plus stable estimates which led to the proof of existence of a unique adapted distributional solution to the degenerate problem. In this paper, following \cite{nikos1}, we are also employing an appropriate vanishing viscosity approximation, but in our case the method has to be modified, as it is not possible to obtain solutions to these approximations by variational methods, when lower order terms are present.

More precisely, the starting point of the method of proof of Theorem \ref{thm:1} is based on the vanishing viscosity approximation of \eqref{eq:1} by the (strictly elliptic) PDE systems

\begin{equation} \label{eq:1s}
\left\{
\begin{gathered}
			\hfil \sum_{\beta, i, j} (\mathbf{A}_{\alpha i \beta j} + \varepsilon \delta_{\alpha \beta} \delta_{ij})\mathrm{D}_{ij}^{2}u^{\varepsilon}_{\beta} + \sum_{\beta, i} \mathbf{B}_{\alpha \beta i}\mathrm{D}_{i}u^{\varepsilon}_{\beta} + \sum_{\beta} \mathbf{C}_{\alpha \beta}u^{\varepsilon}_{\beta} = f_{\alpha}, \hspace{0.1cm} \text{in $\Omega$}, \\
            u^{\varepsilon} = 0, \hspace{0.1cm} \text{on $\partial \Omega$.}
		 \end{gathered}
		 \right.
\end{equation}
However, the variational method doesn't work in our case with lower order terms. We will demonstrate this \emph{formally} in the special case of $\e=0$, as the issue arises if a gradient term is present. Indeed, if we hypothetically defined $L(P, \eta, x)$ (for $P \in \mathbb{R}^{Nn}$, $\eta \in \mathbb{R}^N$ and $x \in \Omega$) as

\begin{equation*}
L(P, \eta, x) \coloneqq \frac{1}{2}\sum_{\alpha, \beta, i, j} \mathbf{A}_{\alpha i \beta j} P_{\alpha i}P_{\beta j} + \sum_{\alpha, \beta, i} \mathbf{B}_{\alpha \beta i}P_{\beta i} \eta_{\alpha} - \frac{1}{2}\sum_{\alpha, \beta} \mathbf{C}_{\alpha \beta} \eta_{\beta}\eta_{\alpha} + \sum_{\beta} f_{\beta}\eta_{\beta}
\end{equation*}
we may compute
\[
\left\{ 
 \ \ 
\begin{split}
L_{P_{\alpha i}} &= \sum_{\beta, i} \mathbf{A}_{\alpha i \beta j}P_{\beta j} + \sum_{\lambda} \mathbf{B}_{\lambda \alpha i} \eta_{\lambda},
\\
L_{\eta_{\alpha}} &= \sum_{\beta, i} \mathbf{B}_{\alpha \beta i} P_{\beta i} - \sum_{\alpha} \mathbf{C}_{\alpha \beta}\eta_{\beta} + f_{\alpha}.
\end{split}
\right.
\]
Therefore,
\[
\sum_{i}\D_i  \big(L_{P_{\alpha i}}(\mathrm{D}u, u, \cdot)\big) = \sum_{\beta, i, j} \mathbf{A}_{\alpha i \beta j}\mathrm{D}^2_{ij}u_{\beta} + \sum_{\lambda, i} \mathbf{B}_{\lambda \alpha i} \mathrm{D}_iu_{\lambda},
\]
and hence, the Euler-Lagrange system of equations
\[
-\sum_{i}\D_i\big(L_{P_{\alpha i}}(\mathrm{D}u, u, \cdot)\big) + L_{\eta_{\alpha}}(\mathrm{D}u, u, \cdot) = 0
\]
becomes
\[
-\sum_{\beta, i, j} \mathbf{A}_{\alpha i \beta j}\mathrm{D}^2_{ij}u_{\beta} - \sum_{\lambda, i} \mathbf{B}_{\lambda \alpha i} \mathrm{D}_iu_\lambda + \sum_{\beta, i} \mathbf{B}_{\alpha \beta i} \mathrm{D}_iu_\beta- \sum_{\alpha} \mathbf{C}_{\alpha \beta}u_{\beta} + f_{\alpha} = 0.
\]
However, under the symmetry assumption that $\mathbf{B}_{\alpha \beta i} = \mathbf{B}_{\beta \alpha i}$, the above system of equations yields
\begin{equation}
\sum_{\beta, i, j} \mathbf{A}_{\alpha i \beta j}\mathrm{D}^2_{ij}u_{\beta} +\sum_{\alpha} \mathbf{C}_{\alpha \beta}u_{\beta} = f_{\alpha} \ \ \text{ in $\Omega$.}
\end{equation}
Therefore, assumption \eqref{as:3} leads to the cancellation of the first-order term of the PDE system. More generally, the following result can be established, which shows the limitations arising in the case there is an explicit gradient term present in the linear system.


\begin{lem} \label{lem:1}
Let $\mathbf{B}: \mathbb{R}^{Nn} \longrightarrow \mathbb{R}^{N}$ be a linear map satisfying the symmetry condition $\mathbf{B}_{\alpha \beta i} = \mathbf{B}_{\beta \alpha i}$. Then, for any $u \in W^{1,2}_0(\Omega, \mathbb{R}^N)$, we have
\begin{equation*}
\int_{\Omega} \sum_{\alpha, \beta, i}\mathbf{B}_{\alpha \beta i} \mathrm{D}_i u_{\beta} u_{\alpha} = 0.
\end{equation*}
\end{lem}

\begin{proof}[Proof of Lemma \ref{lem:1}]
Fix $u \in W^{1,2}_0(\Omega, \mathbb{R}^N)$. Note first that we can rewrite the term $(\mathrm{D}_i u_\beta)u_\alpha$ as 
\[
\mathrm{D}_i(u_\alpha u_\beta) - (\mathrm{D}_i u_\alpha)u_\beta. 
\]
Therefore, we have the idenity
\begin{equation}\label{eq:5new}
\int_{\Omega} \sum_{\alpha, \beta, i}\mathbf{B}_{\alpha \beta i} (\mathrm{D}_i u_{\beta}) u_{\alpha} = \int_{\Omega} \sum_{\alpha, \beta, i}\mathbf{B}_{\alpha \beta i} \mathrm{D}_i(u_{\alpha} u_{\beta}) - \int_{\Omega} \sum_{\alpha, \beta, i}\mathbf{B}_{\alpha \beta i} (\mathrm{D}_i u_{\alpha}) u_{\beta}.
\end{equation}
Integrating by parts the first term on the right-hand side of equation \eqref{eq:5new} leads to

\begin{equation}
\int_{\Omega} \sum_{\alpha, \beta, i}\mathbf{B}_{\alpha \beta i} \mathrm{D}_i(u_{\alpha} u_{\beta}) = \int_{\partial \Omega} \sum_{\alpha, \beta, i}\mathbf{B}_{\alpha \beta i}(u_{\alpha} u_{\beta})\nu_i = 0,
\end{equation}
where the last equality owes to that $u \in W^{1,2}_0(\Omega, \mathbb{R}^N)$. Therefore, we arrive at
\begin{equation} \label{eq:4new}
\int_{\Omega} \sum_{\alpha, \beta, i}\mathbf{B}_{\alpha \beta i} (\mathrm{D}_i u_{\beta}) u_{\alpha} = - \int_{\Omega} \sum_{\alpha, \beta, i}\mathbf{B}_{\alpha \beta i} (\mathrm{D}_i u_{\alpha}) u_{\beta}.
\end{equation}
Under the assumption that $\mathbf{B}_{\alpha \beta i} = \mathbf{B}_{\beta \alpha i}$, we obtain
\[
\int_{\Omega} \sum_{\alpha, \beta, i}\mathbf{B}_{\alpha \beta i} (\mathrm{D}_i u_{\beta}) u_{\alpha} = \int_{\Omega} \sum_{\alpha, \beta, i}\mathbf{B}_{\beta \alpha i} (\mathrm{D}_i u_{\alpha}) u_{\beta}= \int_{\Omega} \sum_{\alpha, \beta, i}\mathbf{B}_{\alpha \beta i} (\mathrm{D}_i u_{\alpha}) u_{\beta},
\]
where the first equality is due to the relabelling of dummy indices. Therefore, by equation \eqref{eq:4new} and the above, we deduce
\begin{equation*}
\int_{\Omega} \sum_{\alpha, \beta, i}\mathbf{B}_{\alpha \beta i} (\mathrm{D}_i u_{\beta}) u_{\alpha} = 0.
\end{equation*}
The conclusion ensues.
\end{proof}

The main difference between the methods followed in \cite{nikos1} in relation to the methods employed herein is that we utilise an appropriate adaptation of the Lax Milgram theorem to the present setting in order to obtain existence of solution to our approximate regularised elliptic problems, which involve lower order terms.

\subsection*{Examples} (i) Lemma \ref{lem:1} is not in general true if $\mathbf{B}$ is not constant. Let
\[
\mathbf{B} \coloneqq e_1 \otimes e_2 \otimes e_1.
\]
It is clear that $\mathbf{B}_{\alpha \beta i} \neq \mathbf{B}_{\beta \alpha i}$. Then, for any $u \in W^{1,2}_0(\Omega, \mathbb{R}^N)$ we have
\begin{equation*}
\sum_{\beta, i} \mathbf{B}_{\alpha \beta i}\mathrm{D}_iu_\beta = \sum_{\beta, i} (e_1)_\alpha (e_2)_\beta (e_1)_i\mathrm{D}_iu_\beta = (e_2 \cdot ((\mathrm{D}u)e_1))e_1 = \mathrm{D}_1u_2 e_1. 
\end{equation*}
Therefore
\begin{equation*}
\int_{\Omega} \sum_{\alpha, \beta, i} \mathbf{B}_{\alpha \beta i}\mathrm{D}_iu_\beta u_\alpha  = \int_{\Omega} \mathrm{D}_1 u_2 u_1 =  -\int_{\Omega} u_2 \mathrm{D}_1u_1 + \int_{\partial \Omega} u_2u_1 \nu_1 =  -\int_{\Omega} u_2 \mathrm{D}_1u_1.
\end{equation*}
It is clear from the above that the right hand side term cannot always vanish for all maps $u \in W^{1,2}_0(\Omega, \mathbb{R}^N)$.

\smallskip

\noindent (ii) This is one of the numerous examples of degenerate elliptic PDE systems which satisfies our assumptions, and hence our result applies. Suppose $\alpha, \beta \in \{1, 2\}$ and $i, j \in \{1, 2\}$. We define 

\begin{equation*}
\mathbf{A}_{\alpha \beta i j} \coloneqq  \delta_{\alpha 1} \delta_{\beta 1}\delta_{i j} = \begin{cases}
\delta_{\alpha \beta} \delta_{ij} &\text{$\alpha = \beta = 1$},\\
0 &\text{$\alpha = 1$, $\beta = 2$},\\
0 &\text{$\alpha = 2$, $\beta = 1$},\\
0 &\text{$\alpha = \beta = 2$}.
\end{cases} 
\end{equation*}

\begin{equation*}
\mathbf{B}_{\alpha \beta i} \coloneqq  \delta_{\alpha 1} \delta_{\beta 1}B_i = \begin{cases}
B_i &\text{$\alpha = \beta = 1$},\\
0 &\text{$\alpha = 1$, $\beta = 2$},\\
0 &\text{$\alpha = 2$, $\beta = 1$},\\
0 &\text{$\alpha = \beta = 2$},
\end{cases} 
\end{equation*}
where $B_i \in \mathbb{R}$, for all $i \in \{1, 2\}$. And
\begin{equation*}
\mathbf{C}_{\alpha \beta} \coloneqq  \delta_{\alpha 1} \delta_{\beta 1}C= \begin{bmatrix}
C & 0 \\
0 & 0
\end{bmatrix}
\end{equation*}
where $C \leq 0$ in $\R$. This leads to the degenerate elliptic PDE system
\begin{align*}
\sum_{\beta, i, j} \delta_{\alpha 1} \delta_{\beta 1}\delta_{i j} \mathrm{D}_{ij}^{2}u_{\beta} + \sum_{\beta, i} \delta_{\alpha 1} \delta_{\beta 1}B_i \mathrm{D}_{i}u_{\beta} + \sum_{\beta}  \delta_{\alpha 1} \delta_{\beta 1}Cu_{\beta} &= f_{\alpha},
\end{align*}
which can be rewritten as
\begin{align*}
\sum_{i} \delta_{\alpha 1} \mathrm{D}_{ii}^{2}u_{1} + \sum_{i} \delta_{\alpha 1} B_i \mathrm{D}_{i}u_{1} + \delta_{\beta 1}Cu_{1} &= f_{\alpha}.
\end{align*}
In particular, this examples shows that, owing to the degeneracy of the coefficients, the problem is solvable only under the necessary compatibility condition $f_2\equiv0$.

\section{Proof of the Main Result}

In this section we establish Theorem \ref{thm:1}.

\begin{proofpart}
We begin by discussing some auxiliary algebraic properties arising due to the degenerate nature of the problem. In view of the definitions \eqref{eq:01}, the properties \eqref{eq:0} and the spectral theorem applied to the symmetric linear map $\mathbf{A} : \mathbb{R}^{Nn} \longrightarrow \mathbb{R}^{Nn}$, imply that there exists a $\nu > 0$ such that

\begin{equation} \label{eq:02}
\Pi \mathbf{A} \Pi = \mathbf{A} =  \mathbf{A} \Pi = \Pi \mathbf{A},
\end{equation}
\begin{equation}  \label{eq:04}
\sum_{\alpha, \beta, i, j} \mathbf{A}_{\alpha i \beta j} Q_{\alpha i} Q_{\beta j} \geq \nu \lvert \Pi Q \rvert^2, \hspace{0.3cm} Q \in \mathbb{R}^{Nn}.
\end{equation}
Identity \eqref{eq:02} states that $\mathbf{A}$ commutes with $\Pi$ on its range. \eqref{eq:04} follows from the fact that the restriction $\mathbf{A} \rvert_{\Pi}$ on the range is invertible. Furthermore, we have that
\begin{equation} \label{eq:4}
\mathbf{X} \in \mathbb{R}^{Nn^2}  \implies \sum_{\alpha, \beta, i, j} (\mathbf{A}_{\alpha i \beta j}\mathbf{X}_{\beta i j})e^{\alpha} \in \Sigma \subseteq \mathbb{R}^{N}.
\end{equation}
Equation \eqref{eq:4} says that the image of the symmetric linear map $\mathbf{A} : \mathbb{R}^{Nn^2} \longrightarrow \mathbb{R}^{N}$ lies completely in $\Sigma$. Every $\mathbf{X} \in \mathbb{R}^{Nn^2}$ can be expressed as
\[
\mathbf{X} = \sum_{A = 1}^{Nn^2} \xi^A \otimes d^A \otimes b^A, \hspace{0.4cm} \xi^A \in \mathbb{R}^N, \hspace{0.4cm}d^A, b^A \in \mathbb{R}^n.
\]
Then, by \eqref{eq:02}, we have 
\begin{equation} \label{eq:ta}
\begin{split}
\eta_{\alpha} \coloneqq \sum_{\beta, i, j} \mathbf{A}_{\alpha i \beta j} \mathbf{X}_{\beta i j} &= \sum_{A = 1}^{Nn^2} \sum_{\beta, i, j} \mathbf{A}_{\alpha i \beta j} \xi^{A}_{\beta}d^{A}_{i}b^{A}_{j} \\&= \sum_{\kappa, k, i}\biggl\{\sum_{A = 1}^{Nn^2} \sum_{\beta, j} \mathbf{A}_{\kappa k \beta j} \xi^{A}_{\beta}b^{A}_{j}\biggl\} \Pi_{\kappa k \alpha i} d^{A}_{i}. \nonumber
\end{split}
\end{equation}
By the definition of $\Sigma$ in \eqref{eq:0}, equation \eqref{eq:ta} says $\eta \coloneqq \sum_{\alpha} \eta_{\alpha}e^{\alpha}$ is a member of $\Sigma$ (hence \eqref{eq:4} follows). Assumption \eqref{as:2} implies that $\mathbf{C}$ commutes with $\Sigma$ on its range and $\mathbf{C}$ maps only from $\Sigma \subseteq \mathbb{R}^N$ to $\Sigma \subseteq \mathbb{R}^N$. Assumption \eqref{as:3} implies that $\mathbf{B}$ maps $\Pi \subseteq \mathbb{R}^{Nn}$ into $\Sigma \subseteq \mathbb{R}^N$.
\end{proofpart}

\begin{proofpart}
Fix $f \in L^2(\Omega,\Sigma)$ and $\varepsilon \in (0,1)$, and consider the regularised Dirichlet problem \eqref{eq:1s}. The weak formulation of this problem for some putative solution $u^\e \in W^{1,2}_0(\Omega, \R^N)$ reads
\begin{equation} \label{eq:3}
\begin{split}
\int_{\Omega} \biggl\{\sum_{\alpha,\beta, i, j} \mathbf{A}^{\varepsilon}_{\alpha i \beta j}\mathrm{D}_{i}u^{\varepsilon}_{\alpha}\mathrm{D}_{j}v_{\beta} - \sum_{\alpha, \beta, i} \mathbf{B}_{\alpha \beta i}\mathrm{D}_{i}u^{\varepsilon}_{\beta}v_{\alpha} - \sum_{\alpha, \beta} \mathbf{C}_{\alpha \beta} & u^{\varepsilon}_{\beta}v_{\alpha} \biggl\}  
\\
&=  -\int_{\Omega} \sum_{\gamma} f_{\gamma}v_{\gamma},
\end{split}
\end{equation}
for all $v \in W^{1,2}_0(\Omega, \R^N)$, where in \eqref{eq:3} we have symbolised for brevity
\begin{equation*}
\mathbf{A}^{\varepsilon}_{\alpha i \beta j} \coloneqq \mathbf{A}_{\alpha i \beta j} + \varepsilon \delta_{\alpha \beta} \delta_{i j}.
\end{equation*}
The first goal is to establish the existence and uniqueness of a weak solution $u^\e \in W^{1,2}_0(\Omega, \R^N)$ to the approximating regularised problem \eqref{eq:1s}. To this aim, we define
\[
\mathrm{E}_\e \ : \ \ W^{1,2}_0(\Om,\R^N) \by  W^{1,2}_0(\Om,\R^N) \larrow \R
\]
by setting
\[
\mathrm{E}_\e(u, v) \coloneqq \int_{\Omega} \biggl\{\sum_{\alpha,\beta, i, j} \mathbf{A}^{\varepsilon}_{\alpha i \beta j}\mathrm{D}_{i}u_{\alpha}\mathrm{D}_{j}v_{\beta} - \sum_{\alpha, \beta, i} \mathbf{B}_{\alpha \beta i}\mathrm{D}_{i}u_{\beta}v_{\alpha} - \sum_{\alpha, \beta} \mathbf{C}_{\alpha \beta}u_{\beta}v_{\alpha} \biggl\}.
\]
Then, equation \eqref{eq:3} can also be written as
\[
\mathrm{E}_\e(u^\e, v) + \langle f,v \rangle = 0,
\]
for all $v \in W^{1,2}_0(\Omega, \R^N)$, where we have symbolised
\[
\langle f,v \rangle :=  \int_{\Omega} \sum_{\gamma} f_{\gamma}v_{\gamma}.
\]
\end{proofpart}

\begin{proofpart}
Note first that $\mathrm{E}_\e(\cdot, \cdot)$ is clearly a bilinear form on $W^{1,2}_0(\Omega, \R^N)$. We will now show that there exists a unique $u^\varepsilon \in W^{1,2}_0(\Omega, \mathbb{R}^N)$ such that 
\[
\mathrm{E}_\e(u^\e, v) = -\langle f,v \rangle,
\]
for all $v \in W^{1,2}_0(\Omega, \R^N)$. Note that $\mathrm{E}_\e$ is {\it not} a symmetric bilinear form (due to its first-order part), hence the Riesz representation theorem cannot be applied on this occasion. Fix $u \in W^{1,2}_0(\Om,\R^N)$. Let us rewrite the bilinear form $\mathrm{E}_\e$ on the diagonal as
\begin{align*}
\mathrm{E}_\e(u, u) &= \int_{\Omega} \sum_{\alpha,\beta, i, j} \mathbf{A}_{\alpha i \beta j}\mathrm{D}_{i}u_{\alpha}\mathrm{D}_{j}u_{\beta} +\varepsilon \int_{\Omega} \lvert \mathrm{D}u \rvert ^2  
\\ 
&\qquad\quad - \int_{\Omega}  \sum_{\alpha, \beta, i} \mathbf{B}_{\alpha \beta i}\mathrm{D}_{i}u_{\beta}u_{\alpha} - \int_{\Omega}  \sum_{\alpha, \beta} \mathbf{C}_{\alpha \beta}u_{\beta}u_{\alpha}.
\end{align*}
We first show that $\mathrm{E}_\e$ is coercive. Recall that, by assumption \eqref{eq:0} we have that $\mathbf{A}:Q\ot Q \geq 0$, for any $Q\in \R^{N\by n}$, and by assumption \eqref{as:2} we similarly have that $\mathbf{C}:\eta\ot \eta \leq 0$, for any $\eta \in \R^N$. These algebraic inequalities imply that the first and last integrals above are non-negative. Further, in view of assumption \eqref{as:3}, by Lemma \ref{lem:1} we have
\[
\int_{\Omega}  \sum_{\alpha, \beta, i} \mathbf{B}_{\alpha \beta i}\mathrm{D}_{i}u_{\beta}u_{\alpha} = 0,
\]
which implies that the integral involving $\mathbf{B}$ vanishes identically. These facts allow us to estimate
\[
\mathrm{E}_\e(u, u) \geq \varepsilon \int_{\Omega} \lvert \mathrm{D}u \rvert ^2 ,
\]
for any $u \in W^{1,2}_0(\Om,\R^N)$. Further, by the (standard) Poincar\'e inequality, there exists $C=C_\e>0$ which allows us to estimate
\begin{align}
\mathrm{E}_\e(u, u) \geq C_\e \| u \| ^2_{W^{1,2}(\Omega)}, \nonumber
\end{align}
for any $u \in W^{1,2}_0(\Om,\R^N)$. This establishes that $E_{\epsilon}$ is indeed coercive. Now, we show that $\mathrm{E}_\e$ is bounded linear (equivalently, continuous) in both variables. To this end, fix $u,v \in W^{1,2}_0(\Om,\R^N)$. By the Cauchy-Schwarz and Poincar\'e inequalities, and by using that $\e \in [0,1]$, we estimate
\begin{align}
\lvert \mathrm{E}_\e(u, v) \rvert &\leq \sum_{\alpha, \beta, i, j} \big| \mathbf{A}_{\alpha i \beta j} \big| \int_{\Omega} \lvert \mathrm{D}u \rvert \lvert \mathrm{D}v \rvert + \varepsilon  \int_{\Omega} \lvert \mathrm{D}u \rvert \lvert \mathrm{D}v \rvert 
\\ 
&\qquad\quad + \sum_{\alpha, \beta, i} \big| \mathbf{B}_{\alpha \beta i} \big| \int_{\Omega} \lvert \mathrm{D}u \rvert \lvert v \rvert + \sum_{\alpha, \beta} | \mathbf{C}_{\alpha \beta}| \int_{\Omega} \lvert u \rvert \lvert v \rvert \nonumber 
\\ 
&\leq c_{1} \| \mathrm{D}u \| _{L^2(\Omega)}\| \mathrm{D}v \| _{L^2(\Omega)} \nonumber 
\\ 
&\qquad\quad + c_{2} \| u \| _{L^2(\Omega)}\| v \| _{L^2(\Omega)} + c_{3} \| \mathrm{D}u \| _{L^2(\Omega)}\| v \| _{L^2(\Omega)} \nonumber 
\\ 
&\leq c_{4} \| \mathrm{D}u \| _{L^2(\Omega)}\| \mathrm{D}v \| _{L^2(\Omega)} \nonumber 
\\ 
&\leq  c_{4} \| u \| _{W^{1,2}(\Omega)}\| v \| _{W^{1,2}(\Omega)}, \nonumber
\end{align}
for some appropriate constants $c_1, c_2, c_3,c_4>0$ depending only the parameters of the Dirichlet problem. This shows that $E_{\epsilon}$ is indeed bounded linear in both variables. Finally, note that $f$ defines a bounded linear functional on $W^{1,2}(\Omega,\R^N)$, since
\begin{align}
\big|  \langle f,v \rangle \big| \leq \| f \| _{L^2(\Omega)}\| v \| _{L^2(\Omega)} \leq \| f \| _{L^2(\Omega)}\| v \| _{W^{1,2}(\Omega)}. \nonumber
\end{align}
In conclusion, all the assumptions of the Lax-Milgram Theorem, are satisfied for the bilinear form $\mathrm{E}_\e$. This yields the existence of a unique $u^\varepsilon \in W^{1,2}_{0}(\Omega, \mathbb{R}^N)$ such that
\[
\mathrm{E}_\e(u^{\varepsilon}, v) + \langle f,v \rangle = 0,
\]
for all $v \in W^{1,2}_0(\Omega,\R^N)$, as claimed. 
\end{proofpart}


\begin{proofpart} Now we derive stable estimates for the approximate solutions $u^\e$, that will allow us to pass to the limit as $\e \to 0$ along a subsequence, to obtain an adapted distributional solution to the original problem. To this aim, fix $\delta > 0$, and recall the compatibility condition for $f$ which states that $\Sigma f = f$ (i.e.\ $f$ is valued into $\Sigma \sub \R^N$), and assumptions  \eqref{eq:0}, \eqref{as:2}, \eqref{as:3}. By Lemma \ref{lem:1} and \eqref{eq:04}, Young's inequality applied to the weak formulation of \eqref{eq:1s} for the admissible choice of $v:=u^\e$, yields
\[
\begin{split}
0 &= \int_{\Omega} \sum_{\alpha,\beta, i, j} \mathbf{A}_{\alpha i \beta j}\mathrm{D}_{i}u^{\varepsilon}_{\alpha}\mathrm{D}_{j}u^\varepsilon_{\beta} +\varepsilon \int_{\Omega} \lvert \mathrm{D}u^\varepsilon \rvert ^2  
\\
&\qquad\quad - \int_{\Omega}  \sum_{\alpha, \beta, i} \mathbf{B}_{\alpha \beta i}\mathrm{D}_{i}u^{\varepsilon}_{\beta}u^\varepsilon_{\alpha} - \int_{\Omega}  \sum_{\alpha, \beta} \mathbf{C}_{\alpha \beta}u^{\varepsilon}_{\beta}u^\varepsilon_{\alpha} + \int_{\Omega} \sum_{\gamma} f_{\gamma}u^\varepsilon_{\gamma} \nonumber 
\\ 
&\geq \nu \int_{\Omega} \lvert \mathrm{D}u^\varepsilon \rvert^2 - \int_{\Omega} \sum_{\gamma} f_{\gamma}(\Sigma u^\varepsilon)_{\gamma} \nonumber
 \\
  &\geq \nu \int_{\Omega} \lvert \Pi \mathrm{D}u^\varepsilon \rvert^2 -\frac{1}{4 \delta} \int_{\Omega} \lvert f \rvert^2 - \delta \int_{\Omega} \lvert \Sigma u^\varepsilon \rvert^2. \nonumber
\end{split}
\]
Therefore, we have
\begin{equation} \label{eq:prop1}
\nu  \| \Pi \mathrm{D}u^\varepsilon \| ^2_{L^2(\Omega)}  \leq \frac{1}{4 \delta } \| f \| ^2_{L^2(\Omega)} +   \de \| \Sigma u^\varepsilon \| ^2_{L^2(\Omega)}.
\end{equation}
We now select the parameter $\de>0$ by setting
\[
\delta := \frac{\nu}{2C}, 
\]
where $C$ is the constant appearing in the degenerate Poincar\'e inequality (Proposition \ref{prop:1}). As a result, inequality \eqref{eq:prop1} implies
\[
\begin{split}
\frac{1}{4 \delta} \| f \| ^2_{L^2(\Omega)} &\geq \nu \| \Pi \mathrm{D}u^\varepsilon \| ^2_{L^2(\Omega)} - \delta \| \Sigma u^\varepsilon \| ^2_{L^2(\Omega)} \\ &\geq \nu \| \Pi \mathrm{D}u^\varepsilon \| ^2_{L^2(\Omega)} - C\delta \| \Pi \mathrm{D}u^\varepsilon \| ^2_{L^2(\Omega)} \nonumber \\ &= \nu \| \Pi \mathrm{D}u^\varepsilon \| ^2_{L^2(\Omega)} - \frac{\nu}{2} \| \Pi \mathrm{D}u^\varepsilon \| ^2_{L^2(\Omega)} \nonumber \\ &= \frac{\nu}{2} \| \Pi \mathrm{D}u^\varepsilon \| ^2_{L^2(\Omega)} \nonumber.
\end{split}
\]
Hence, we have the uniform estimate
\begin{equation} \label{eq:prop2}
 \| \Pi \mathrm{D}u^\varepsilon \| ^2_{L^2(\Omega)} \leq \frac{C}{\nu^2}\| f \| ^2_{L^2(\Omega)} .
\end{equation}
Equation \eqref{eq:prop2} shows that $\Pi \mathrm{D}u^\varepsilon$ is uniformly bounded in $L^2(\Omega;\Pi)$. Further, again by Proposition \ref{prop:1}, $\Sigma u^\varepsilon$ is also uniformly bounded in $L^2(\Omega;\Si)$. By standard weak compactness properties, there exists a sequence of indices $(\e_k)_1^\infty$ and mappings $u \in L^2(\Omega, \Sigma)$ and $U \in L^2(\Omega, \Pi)$ such that 
\begin{equation} \label{eq:20B}
\left\{ \ \ 
\begin{split}
&\Sigma u^{\varepsilon} \weak u, \hspace{0.2cm} & \text{in $L^2(\Omega, \Sigma)$}, 
\\
&\Pi \mathrm{D}u^{\varepsilon} \weak U, \hspace{0.2cm} & \text{in $L^2(\Omega, \Pi)$}
\end{split}
\right.
\end{equation}
as $k \rightarrow \infty$, along the sequence $\varepsilon_k \rightarrow 0$. 
\end{proofpart}

\begin{proofpart}
We now show that the mapping $u \in L^2(\Omega, \Sigma)$ obtained in the previous step is actually an adapted distributional solution to the Dirichlet problem \eqref{eq:1}. It is perhaps worth noting that, due to the degenerate nature of the problem, we cannot merely use the Closed Graph theorem to ascertain ``$U=\D u$", as we only have control of the projection of $u$ on the subspace  $\Si \sub \R^N$, and of $\D u$ only on the subspace $\Pi \sub \R^N$. Recall that, for any $\varepsilon \in (0,1)$, $u^{\varepsilon} \in W^{1,2}_0(\Om,\R^N)$ is a weak solution to \eqref{eq:1s}. Fix a test mapping $\phi \in C^\infty_c(\Om,\Si)$. Then, we have
\[
\int_\Om \bigg\{\sum_{\al, \beta, i, j} \mathbf{A}^\e_{\alpha i \beta j} \mathrm{D}_{i}u^{\varepsilon}_{\beta}\D_i \phi_\al - \sum_{\al,\beta, i} \mathbf{B}_{\alpha \beta i}\mathrm{D}_{i}u^{\varepsilon}_{\beta} \phi_\al - \sum_{\al, \beta} \mathbf{C}_{\alpha \beta}u^{\varepsilon}_{\beta}\phi_\al +\sum_{\al}f_{\alpha}\phi_\al \bigg\}=0.
\]
By recalling our notation
\[
\mathbf{A}^\e_{\alpha i \beta j} =\mathbf{A}_{\alpha i \beta j} +\e \de_{\al\be}\de_{ij},
\]
an integration by parts gives
\[
\begin{split}
\int_\Om \bigg\{\sum_{\al, \beta, i, j} \mathbf{A}_{\alpha i \beta j} \D^2_{ij} \phi_\al u^{\varepsilon}_{\beta} -  \sum_{\al,\beta, i} \mathbf{B}_{\alpha \beta i} \D_i \phi_\al u^{\varepsilon}_{\beta}  + & \sum_{\al, \beta} \mathbf{C}_{\alpha \beta} \phi_\al  u^{\varepsilon}_{\beta} \bigg\}
\\
&= \int_\Om \sum_{\al} \Big\{ f_{\alpha}\phi_\al -\e  u^{\varepsilon}_{\al}\D^2_{ii} \phi_\al \Big\}.
\end{split}
\]
Note now that $\Sigma \phi = \phi$, and that by  \eqref{eq:0}, \eqref{as:2}, \eqref{as:3} and \eqref{eq:4}, we have 
\[
\left\{ 
 \ \ 
 \begin{split}
\sum_{\al, i, j}  \mathbf{A}_{\alpha i \beta j} \mathrm{D}^2_{ij} \phi_{\al} &= \sum_{\al, \la, i, j} \mathbf{A}_{\al i \la j} \mathrm{D}^2_{ij} \phi_{\al} \Sigma_{\lambda \be},
\\
\sum_{\al, i}  \mathbf{B}_{\alpha \beta i}\mathrm{D}_i \phi_{\al}& = \sum_{\al, \lambda, i}  \mathbf{B}_{\al \la i}\mathrm{D}_i \phi_{\al} \Sigma_{\lambda \be},
\\
\sum_{\al}  \mathbf{C}_{\alpha \beta}\phi_\al &= \sum_{\al, \lambda}  \mathbf{C}_{\al \lambda}\phi_\al \Sigma_{\lambda \be}.
\end{split}
\right.
\]
These identities imply
\[
\begin{split}
\int_{\Omega} \sum_{\kappa, \alpha, \beta, i, j} & (\mathbf{A}_{\alpha i \beta j} \mathrm{D}^2_{ij}\phi_{\alpha})(\Sigma_{\beta \kappa}u^{\varepsilon}_{\kappa}) - \int_{\Omega} \sum_{\kappa, \alpha, \beta, i} (\mathbf{B}_{\alpha \beta i} \mathrm{D}_{i}\phi_{\alpha})(\Sigma_{\beta \kappa}u^{\varepsilon}_{\kappa}) 
\\
&+  \int_{\Omega} \sum_{\kappa, \alpha, \beta} (\mathbf{C}_{\alpha \beta}\phi_{\alpha})(\Sigma_{\beta \kappa}u^{\varepsilon}_{\kappa}) = \int_{\Omega} \sum_{\gamma} f_{\gamma} \phi_{\gamma} - \varepsilon \int_{\Omega} \sum_{\lambda, \kappa} \Delta \phi_{\lambda}(\Sigma_{\lambda \kappa} u^{\varepsilon}_{\kappa}).
\end{split}
\]
By letting $\varepsilon = \varepsilon_k \rightarrow 0$, the convergences of \eqref{eq:20B} imply that $u \in L^2(\Omega, \Sigma)$ and that it is an adapted distributional solution to the PDE system
\begin{equation*} 
\sum_{\beta, i, j} \mathbf{A}_{\alpha i \beta j}\mathrm{D}_{ij}^{2}u_{\beta} + \sum_{\beta, i} \mathbf{B}_{\alpha \beta i}\mathrm{D}_{i}u_{\beta} + \sum_{\beta} \mathbf{C}_{\alpha \beta}u_{\beta} = f_{\alpha} \ \text{ in }\Om. 
\end{equation*}
Further, Proposition \ref{prop:2} implies the existence of a degenerate partial trace operator on $\p\Om$. Since $u^\e \in W^{1,2}_0(\Om,\R^N)$ for any $\e \in(0,1)$, this implies that $u$ is measurable with respect to the $(n-1)$-Hausdorff measure and well defined up to nullsets on the boundary of any strictly convex open subset of $\Om$. Additionally, $u=0$ $\mH^{n-1}$-a.e.\ on $\p\Om$. Hence, the boundary condition is satisfied as well.
\end{proofpart}

\begin{proofpart} Now we discuss the partial regularity property of the adapted distributional solution $u$, constructed in the previous step. Recall that by \eqref{eq:20B}, we have
\[
\left\{ \ \ 
\begin{split}
&\Sigma u^{\varepsilon} \weak u, \hspace{0.2cm} & \text{in $L^2(\Omega, \Sigma)$}, 
\\
&\Pi \mathrm{D}u^{\varepsilon} \weak U, \hspace{0.2cm} & \text{in $L^2(\Omega, \Pi)$},
\end{split}
\right.
\]
as $\e \to 0$, along an infinitesimal sequence $(\e_k)_1^\infty$. Recalling that by assumption $\Pi$ is spanned by rank-one matrices, we may select $\eta \ot a \in \Pi \setminus \{0\}$. Then, the above implies
\[
\left\{ \ \ 
\begin{split}
&\eta \cdot u^{\varepsilon} \weak \eta \cdot u, \hspace{0.2cm} & \text{in $L^2(\Omega, \Sigma)$}, 
\\
&\mathrm{D}_a (\eta \cdot u^{\varepsilon}) \weak (\eta \ot a) : U, \hspace{0.2cm} & \text{in $L^2(\Omega, \Pi)$},
\end{split}
\right.
\]
as $\e \to 0$ along the infinitesimal sequence  $(\e_k)_1^\infty$.  By passing to the limit as $k\to\infty$, the Closed Graph theorem then implies that the $a$-directional derivative $\mathrm{D}_a (\eta \cdot u)$ of the projection $\eta \cdot u$ exists in $L^2(\Om)$, and can be represented through contractions of the mapping $U \in L^2(\Omega, \Pi)$ with rank-one directions $\eta \ot a \in \Pi$:
\[
\mathrm{D}_a (\eta \cdot u) = (\eta \ot a) : U \, \in \, L^2(\Om).
\]
\end{proofpart}

\begin{proofpart}
We finally demonstrate uniqueness of the adapted distributional solution to the problem, which in particular implies independence of the vanishing viscosity approximation method it was constructed. For the sake of contradiction, suppose the problem has two adapted distributional solutions, $u$ and $v$ in $\mathscr{D}'(\Omega, \Sigma)$. By the previous, we also have that  $u,v \in L^2(\Om,\Si)$. Then, $w \coloneqq u - v \in L^{2}(\Omega, \Sigma)$ is also an adapted distributional  solution to
\[
\left\{ \ \
\begin{split}
&\sum_{\beta = 1}^{N}\sum_{i, j = 1}^{n} \mathbf{A}_{\alpha i \beta j}\mathrm{D}_{ij}^{2}w_{\beta} + \sum_{\beta = 1}^{N} \sum_{i=1}^{n} \mathbf{B}_{\alpha \beta i}\mathrm{D}_{i}w_{\beta} + \sum_{\beta = 1}^{N} \mathbf{C}_{\alpha \beta}w_{\beta} = 0,  & \text{ in $\Omega$}, 
\\ 
&w = 0, & \text{ on $\partial \Omega$},
\end{split}
\right.
\]
in $\mathscr{D}'(\Omega, \Sigma)$. Let $w$ be extended by zero (on $\mathbb{R}^{n} \setminus \Omega$) to a function in $L^{2}(\mathbb{R}^{n} , \Sigma)$, also denoted by $w$. Let $(\eta^\e)_{\e>0}$ be the standard mollifier (as e.g.\ in \cite{EG}). For any $\varepsilon > 0$, the mollified solution ${w_{\varepsilon}} := w \ast \eta^{\varepsilon}$ satisfies that ${w_{\varepsilon}} \in C^{\infty}_{c}(\mathbb{R}^{n}, \Sigma)$ and also
\[
\sum_{\beta = 1}^{N}\sum_{i, j = 1}^{n} \mathbf{A}_{\alpha i \beta j}\mathrm{D}_{ij}^{2}({w_{\varepsilon}})_{\beta} + \sum_{\beta = 1}^{N} \sum_{i=1}^{n} \mathbf{B}_{\alpha \beta i}\mathrm{D}_{i}({w_{\varepsilon}})_{\beta} + \sum_{\beta = 1}^{N} \mathbf{C}_{\alpha \beta}({w_{\varepsilon}})_{\beta} = 0 \ \text{ in $\mathbb{R}^n.$}
\]
Snce $w_\e$ is compactly supported, integration by parts and the application of inequality \eqref{eq:04} leads to
\[
\begin{split}
\| \Pi \mathrm{D} {w_{\varepsilon}} \| _{L^{2}(\mathbb{R}^{n})}^{2} 
&\leq 
\frac{1}{\nu} \int_{\mathbb{R}^n}\biggl\{\sum_{\alpha, i, \beta, j} \mathbf{A}_{\alpha i \beta j}\mathrm{D}_{i}({w_{\varepsilon}})_{\alpha}\mathrm{D}_{j}({w_{\varepsilon}})_{\beta} \bigg\}
\\
&\leq \int_{\mathbb{R}^n}\biggl\{ \sum_{\alpha, \beta, i} \mathbf{B}_{\alpha \beta i}\mathrm{D}_{i}({w_{\varepsilon}})_{\beta}({w_{\varepsilon}})_{\alpha}  + \sum_{\alpha, \beta} \mathbf{C}_{\alpha \beta}({w_{\varepsilon}})_{\beta}({w_{\varepsilon}})_{\alpha} \biggl\}.
\end{split}
\]
By using \eqref{as:2} and Lemma \ref{lem:1}, this further leads to
\[
\| \Pi \mathrm{D} {w_{\varepsilon}} \| _{L^{2}(\mathbb{R}^{n})}^{2} \leq 0.
\]
Therefore, we have
\[
\| \Pi \mathrm{D} {w_{\varepsilon}} \| _{L^{2}(\mathbb{R}^{n})}^{2} = 0.
\]
Further, from the degenerate partial Poincar\'e inequality given in Proposition \ref{prop:1}, we have

\begin{equation*}
0 \leq \| \Sigma {w_{\varepsilon}} \| _{L^{2}(\mathbb{R}^{n})} \leq C\| \Pi \mathrm{D}{w_{\varepsilon}} \| _{L^{2}(\mathbb{R}^{n})} = 0.
\end{equation*}
Hence we obtain
\[
\| \Sigma {w_{\varepsilon}} \| _{L^{2}(\mathbb{R}^{n})} = 0.
\]
Since  $w$ is valued into $\Sigma \sub \R^N$, it follows that $w^\varepsilon$ is also valued into $\Sigma$. Therefore, ${w_{\varepsilon}} = \Sigma  {w_{\varepsilon}} = 0$ on $\mathbb{R}^{n}$. By letting $\varepsilon \rightarrow 0$, we obtain that $w \equiv 0$, yileding that $u\equiv v$. Hence, the adapted distriburtional solution is unique.  
\end{proofpart}

\noi The proof of the theorem is now complete. \qed

\medskip

\noi {\bf Declarations.} The authors have no conflict of interest to declare. Further, no data were generated or in any way involved in this work. Finally, no AI assistance has been utilised in the development of this work.

\medskip

\end{document}